\documentclass[12pt]{article}

\usepackage{epstopdf}
\usepackage{latexsym, amsmath, amscd,amsthm}
\usepackage{amssymb}
\usepackage{latexsym, amsmath, amscd,amsthm,booktabs,array,url}
\usepackage{graphicx,epsf}

\newtheorem{theorem}{Theorem}[section]
\newtheorem{corollary}[theorem]{Corollary}
\newtheorem{proposition}[theorem]{Proposition}
\newtheorem{lemma}[theorem]{Lemma}

\theoremstyle{definition}

\newtheorem{remark}[theorem]{Remark}

\begin{document}

\title{Classification of graphs by Laplacian eigenvalue distribution and independence number}
\author{Jinwon Choi\footnotemark[1],\, Sunyo Moon\footnotemark[1],\,  
and Seungkook Park\footnotemark[1]} 

\date{}
\renewcommand{\thefootnote}{\fnsymbol{footnote}}
\footnotetext[1]{Department of Mathematics and Research Institute of Natural Sciences, Sookmyung Women's University, Seoul, Republic of Korea(\{jwchoi, symoon, skpark\}@sookmyung.ac.kr)}
\renewcommand{\thefootnote}{\arabic{footnote}}

\maketitle

\begin{abstract}
Let $m_GI$ denote the number of Laplacian eigenvalues of a graph $G$ in an interval $I$ and let $\alpha(G)$ denote the independence number of $G$. In this paper, we determine the classes of graphs that satisfy the condition $m_G[0,n-\alpha(G)]=\alpha(G)$ when $\alpha(G)= 2$ and $\alpha(G)= n-2$, where $n$ is the order of $G$. When $\alpha(G)=2$, $G \cong K_1 \nabla K_{n-m} \nabla K_{m-1}$ for some $m \geq 2$.
When $\alpha(G)=n-2$,
there are two types of graphs $B(p,q,r)$ and $B'(p,q,r)$ of order $n=p+q+r+2$, which we call the binary star graphs.
Also, we show that the binary star graphs with $p=r$ are determined by their Laplacian spectra.
\end{abstract}

\section{Introduction}

Let $G=(V(G),E(G))$ be a simple graph with vertex set $V(G)$ and edge set $E(G)$.
For a given graph $G$ of order $n$, the \textit{Laplacian matrix} $L(G)$ of $G$ is defined as $L(G)=D(G)-A(G)$, where $D(G)$ is the diagonal matrix of vertex degrees and $A(G)$ is the adjacency matrix of $G$.
Note that $L(G)$ is symmetric and positive semidefinite.
For a matrix $M$, we let $\mu(M,x)=\det(xI-M$) be the characteristic polynomial of $M$ and we denote by $\mu(G,x)$ the characteristic polynomial of the Laplacian matrix $L(G)$.  The roots of $\mu(G,x)$ are called the {\it Laplacian eigenvalues} of the graph $G$.
We denote the Laplacian eigenvalues of $G$ by $$0=\lambda_n(G)\leq \lambda_{n-1}(G) \leq \cdots \leq \lambda_1(G).$$ The multiset of all the Laplacian eigenvalues of $G$ is called the {\it Laplacian spectrum} of $G$.

The relation between the graph parameters and the distribution of the Laplacian eigenvalues has been studied by many researchers. In \cite{Wang}, Wang et al. showed that $m_G(n-1,n] \leq \chi(G)-1$, where $\chi(G)$ is the chromatic number of $G$. In \cite{Hedetniemi}, Hedetniemi et al. proved that $m_G[0,1] \leq \gamma$, where $\gamma$ is the domination number of $G$. Recently, Ahanjideh et al.\cite{Ahanjideh} gave a relation between the number of Laplacian eigenvalues and the independent number of a graph. They showed that a connected graph $G$ of order $n$ satisfies $\alpha(G) \leq m_G[0,n-\alpha(G)]$. It would be an interesting problem to characterize all the connected graphs for which the equality holds, that is, $m_G[0,n-\alpha(G)]=\alpha(G)$. For $\alpha(G)=1$, the complete graph $K_n$ of order $n$ is the only graph satisfying $m_G[0,n-1]=1$ and for $\alpha(G)=n-1$, the star graph $S_n$ of order $n$ is the only graph satisfying $m_G[0,1]=n-1$. In this paper, we determine the classes of connected graphs that satisfy the condition $m_G[0,n-\alpha(G)]=\alpha(G)$ when $\alpha(G)= 2$ and $\alpha(G)= n-2$. When $\alpha(G)=2$, the graphs satisfying the condition are $K_1 \nabla K_{n-m} \nabla K_{m-1}$ for some $m \geq 2$. When $\alpha(G)=n-2$,
there are two types of graphs $B(p,q,r)$ and $B'(p,q,r)$ of order $n=p+q+r+2$ with some restrictions. We will call these graphs the binary star graphs. We also prove that the binary star graphs $B(p,q,p)$ and $B'(p,q,p)$ are determined by their Laplacian spectra.

The organization of the paper is as follows. In section 2, we give known definitions and formulas for the characteristic polynomial of the Laplacian matrix of some graphs. In section 3, we classify the graphs that satisfy $m_G[0,n-\alpha(G)]=\alpha(G)$ for $\alpha(G)= 2$ and $\alpha(G)= n-2$. In section 4, we show that the binary star graphs
$B(p,q,p)$ and $B'(p,q,p)$ of order $n=2p+q+2$ are determined by their Laplacian spectra.

\section{Preliminaries}

In this section, we give notations and collect known definitions. We also introduce some properties on the Laplacian eigenvalues and formulas for the characteristic polynomial of the Laplacian matrix of some graphs.
The path graph, cycle graph, star graph, complete graph, and the complete multipartite graph with part sizes $m_1,\ldots, m_t$, where all graphs are of order $n$, are denoted by $P_n$, $C_n$, $S_n$, $K_n$, $K_{m_1,\ldots, m_t}$, respectively.
If the Laplacian eigenvalues of $G$ are $0=\lambda_n(G)\leq \lambda_{n-1}(G) \leq \cdots \leq \lambda_1(G)$ then the eigenvalues of the complement graph $\bar{G}$ of $G$ are $$0=\lambda_n(\bar{G})\leq n-\lambda_1(G) \leq n-\lambda_{2}(G)\leq \cdots \leq n-\lambda_{n-1}(G).$$ In \cite{Fiedler}, Fiedler called the second smallest Laplacian eigenvalue of $G$ the {\it algebraic connectivity} of the graph of $G$, which is a measure of connectivity of $G$. It is well known that $\lambda_{n-1}(G)=0$ if and only if $G$ is disconnected.

The \textit{disjoint union} of the graphs $G_1$ and $G_2$ is denoted by $G_1 \cup G_2$.
Let $G_1, G_2, \ldots, G_k$ be a sequence of pairwise disjoint graphs.
Denote by $G_1 \nabla G_2 \nabla \cdots \nabla G_k$ the graph obtained from $G_1, G_2, \ldots G_k$ by adding all edges $uv$ with $u \in V(G_i)$ and $v \in V(G_{i+1})$. In particular, $G_1 \nabla G_2$ is the \textit{join} of $G_1$ and $G_2$.

\begin{lemma}[{\cite[Theorem 3.7]{Mohar}}]\label{char-join}
  Let $G_1$ and $G_2$ be graphs of order $n_1$ and $n_2$, respectively. Then
  \begin{equation*}
    \mu(G_1\nabla G_2, x) =\frac{x(x-n_1-n_2)}{(x-n_1)(x-n_2)} \mu(G_1,x-n_2)\mu(G_2,x-n_1).
  \end{equation*}
\end{lemma}
\begin{lemma}[{\cite[Theorem 4.1]{Grone}}]\label{interlace}
  Let $G$ be a graph of order $n$ and let $e$ be an edge of $G$. Suppose that $G'$ is the subgraph obtained from $G$ by deleting the edge $e$. Then
  \begin{equation*}
    0=\lambda_{n}(G')\leq \lambda_{n}(G) \leq \lambda_{n-1}(G')\leq \lambda_{n-1}(G) \leq \cdots \leq
    \lambda_{1}(G')\leq \lambda_{1}(G).
  \end{equation*}
\end{lemma}

A vertex of degree 1 is called a {\it pendant vertex} and the edge attached to a pendant vertex is called a {\it pendant edge}.
In \cite{Faria}, Faria defined a {\it pendant star} of a graph as a maximal subgraph formed by pendant edges all incident with the same vertex(the center of the pendant star).
The number of its pendant vertices minus one is called the {\it degree of a pendant star}.
The sum of the degree of all pendant stars is called the {\it star degree} of graph.
\begin{lemma}[\cite{Faria}]\label{stardeg}
  Let $G$ be a graph with star degree $p$, then the multiplicity of $1$ as a root of $\mu(G,x)$ is greater than or equal to $p$.
\end{lemma}

\begin{figure}[h!t!]
  \centering
  \includegraphics[width=4cm]{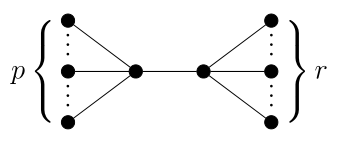}
  \caption{Double star $S(p,r)$}\label{Fig:dobulestar}
\end{figure}

The {\it double star} $S(p,r)$ is a graph obtained from two star graphs $S_{p+1}$ and $S_{r+1}$ joining their central vertices.
In \cite{Grone-O}, Grone and Merris gave the characteristic polynomial of the Laplacian matrix of the double star graph.
\begin{lemma}[{\cite[Proposition 1]{Grone-O}}]\label{char:d-star}
  Fix $n=p+r+2$. Then the characteristic polynomial of $L(S(p,r))$ is
  \begin{equation*}
    x(x-1)^{n-4}(x^3-(n+2)x^2+(2n+pr+1)x-n).
  \end{equation*}
\end{lemma}

\section{Laplacian eigenvalue distribution}

In \cite{Ahanjideh}, Ahanjideh et al. investigated a relation between the number of Laplacian eigenvalues in an interval and the independent number of a graph. They showed that
a connected graph $G$ of order $n$ satisfies $\alpha(G) \leq m_G[0,n-\alpha(G)]$.
Here, we give another proof.
\begin{theorem}
  Let $G$ be a graph of order $n$, then $\alpha(G) \leq m_G[0,n-\alpha(G)]$.
\end{theorem}
\begin{proof}
  The complement graph $\bar{G}$ of $G$ contains the complete graph $K_{\alpha(G)}$.
  Since Laplacian eigenvalues of $K_{\alpha(G)}$ are $0$ and $\alpha(G)$ with the multiplicity ${\alpha(G)}-1$, by Lemma \ref{interlace},
  $$m_{\bar{G}}[{\alpha(G)},n]\geq {\alpha(G)} -1.$$
  Since $\lambda_i(\bar{G})=n-\lambda_{n-i}(G)$ for $1 \leq i\leq n-1$,  we have $m_G[0,n-{\alpha(G)}] \geq {\alpha(G)}$.
\end{proof}

Now, we focus on the graphs $G$ for which $m_G[0,n-\alpha(G)]=\alpha(G)$.
The complete graph $K_n$ is the only graph satisfying $m_G[0,n-\alpha(G)] =\alpha(G)$ when $\alpha(G)=1$ and the star graph $S_n$ is the only graphs satisfying $m_G[0,n-\alpha(G)] =\alpha(G)$ when $\alpha(G)=n-1$ \cite[Theorem 5.2]{Ahanjideh}.
In this section, we consider the cases $\alpha(G)=2$ and $\alpha(G)=n-2$.
For later use, we give the Laplacian spectra of graphs in Figure \ref{Fig:graphs}.
\begin{figure}[h!t!]
  \centering
  \includegraphics[width=12cm]{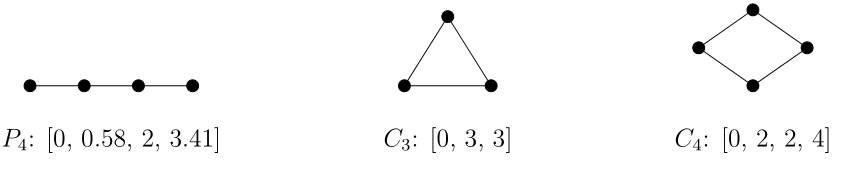}
  \caption{Graphs with their Laplacian spectra}\label{Fig:graphs}
\end{figure}

\begin{theorem}
  Let $G$ be a graph of order $n$ with $\alpha(G)=2$.
  Then $m_G[0,n-2]=2$ if and only if $G \cong K_1 \nabla K_{n-m} \nabla K_{m-1}$ for some $m \geq 2$.
\end{theorem}
\begin{proof}
  Suppose that $m_G[0,n-2]=2$.
  Then the complement graph $\bar{G}$ of $G$ satisfies $m_{\bar{G}}[2,n]=1$.
  By Lemma \ref{interlace}, $\bar{G}$ does not contain a subgraph isomorphic to $P_4$ or $C_3$, because they have two Laplacian eigenvalues greater than or equal to $2$.
  Thus $\bar{G}$ is a disjoint union of trees of diameter less than equal to $2$.
  Thus we have
  \begin{equation*}
    \bar{G} \cong S_{m_1}\cup \cdots \cup S_{m_k} \cup (n-\sum_{i=1}^{k}m_i)K_1 ~~\text{for $m_i \geq 2$.}
  \end{equation*}
  Suppose that $\bar{G}$ contains two star graphs of order at least $2$.
  Since $S_m$ has the Laplacian spectrum $[0, 1^{m-2}, m]$, by Lemma \ref{interlace}, we have $m_{\bar{G}}[2,n]>1$,  which is a contradiction.
  Hence $\bar{G}$ is isomorphic to $S_m \cup (n-m)K_1$ for some $m \geq 2$.
  Since $\bar{S}_m=K_{m-1} \cup K_1$,
  we have
  \begin{equation*}
    G \cong (K_1 \cup K_{m-1}) \nabla K_{n-m} = K_1 \nabla K_{n-m} \nabla K_{m-1}.
  \end{equation*}
  Note that, by Lemma \ref{char-join}, we have
  \begin{equation*}
    \mu(G,x)=x(x-(n-m))(x-(n-1))^{m-2}(x-n)^{n-m}
  \end{equation*}
  Therefore, $m_G[0,n-2]=2$.
\end{proof}

For a connected graph $G$ of order $n \geq 3$, there are two types of graphs with $\alpha(G)=n-2$.
Let $V(G)\setminus\{u,v\}$ be a maximal independent set of $G$.
First, suppose that $u$ and $v$ are not adjacent.
We decompose $V(G)\setminus\{u,v\}$ into three disjoint sets as follows:
\begin{gather*}
  N_G(u)\cap N_G(v) = \{w_1,\ldots,w_q\}, \\
  N_G(u)\setminus (N_G(u)\cap N_G(v))=\{u_1,\ldots,u_p\}, \text{~~and} \\
  N_G(v)\setminus (N_G(u)\cap N_G(v))=\{v_1,\ldots,v_r\}.
\end{gather*}
Let $S_{p+q+1}$ and $S_{q+r+1}$ be star graphs with central vertices $u$ and $v$, respectively.
Then $G$ is given by identifying noncentral $q$ vertices of two star graphs $S_{p+q+1}$ and $S_{q+r+1}$ as in Figure \ref{Fig:b-star}(a).
We denote this type of graph $G$ by $B(p,q,r)$.

When $u$ and $v$ are adjacent, the second type of graph $G$ is obtained from $B(p,q,r)$ by adding an edge between $u$ and $v$ as in Figure \ref{Fig:b-star}(b).
We denote this graph by $B'(p,q,r)$.
We call such graphs {\it binary star graphs.}
Throughout this paper, we assume $p \geq r$.
Note that $B(0,1,0)$ and $B'(p,0,0)$ are star graphs, whose independence number is $n-1$.

\begin{figure}
  \centering
  \includegraphics[width=13cm]{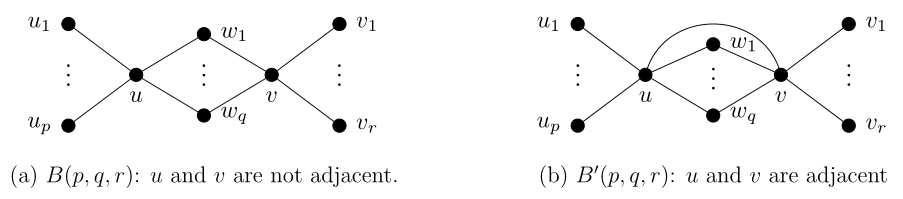}
  \caption{Binary star graphs}\label{Fig:b-star}
\end{figure}

It is well known that for $n\times n$ matrix $M$,
\begin{equation*}
  \det(xI-M)=x^n-a_1x^{n-1}+\cdots+(-1)^{n-1}a_{n-1}x+(-1)^na_n,
\end{equation*}
where $a_k$ is the sum of all the $k\times k$  principal minors of $M$ for $k=1,\ldots,n$.
The following lemma is an analogue of Lemma \ref{stardeg}.
\begin{lemma}\label{deg2}
  Let $G$ be a graph of order $n$ and let $u,v$ be two vertices in $V(G)$.
  If $N=\{w \in V(G) \,|\, N_G(w)=\{u,v\}\}$, then $2$ is a root of the characteristic polynomial of $L(G)$ with multiplicity greater than or equal to $|N|-1$.
\end{lemma}
\begin{proof}
  Let $N=\{w_1,w_2,\ldots,w_q\}$.
  It suffices to show that $0$ is a root of the characteristic polynomial of $L(G)-2I$ with multiplicity greater than or equal to $q-1$.
  The matrix $L(G)-2I$ has the form as in Figure \ref{Fig:L:deg}.
  \begin{figure}[h!]
    \centering
    \includegraphics[width=6.5cm]{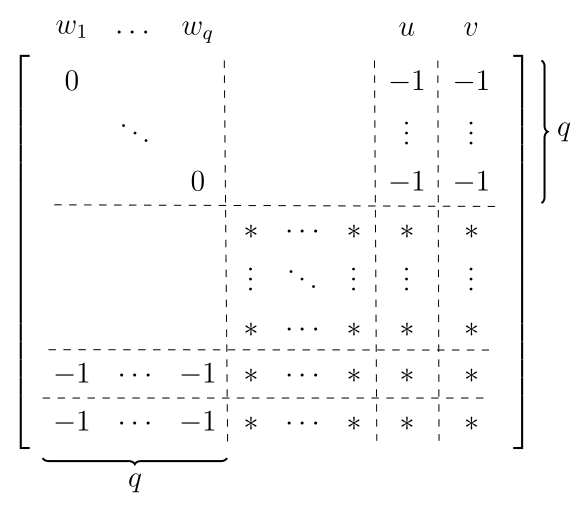}
    \caption{The matrix $L(G)-2I$ }\label{Fig:L:deg}
  \end{figure}

  Let $n-(q-2) \leq k \leq n$.
  Since the first $q$ rows are identical, all the $k \times k$ principal minors of $L(G)-2I$ are zero.
  Hence the coefficients of $x^{n-k}$ in the characteristic polynomial of $L(G)-2I$ are zeros for all $k$. In other words, $0$ is a root of the characteristic polynomial of $L(G)-2I$ with multiplicity greater than or equal to $q-1$.
\end{proof}
\begin{remark}
  By the same method, one can show that the multiplicity of $2$ as a root of the characteristic polynomial of $L(G)$ is greater than or equal to
  \begin{equation*}
    \sum_{u,v \in V(G)} \bigg( \big|\{ w \in V(G) | N_G(w)=\{u,v\} \}\big|-1 \bigg).
  \end{equation*}
\end{remark}

The next theorems give the characteristic polynomial of the Laplacian matrix for the binary star graphs.
\begin{theorem}\label{char:B}
  Let $G$ be the binary star graph $B(p,q,r)$ of order $n=p+q+r+2$. Then the characteristic polynomial of $L(G)$ is
  \begin{equation*}
    \mu(G,x)=x(x-1)^{p+r-2}(x-2)^{q-1}(x^4-a_1x^3+a_2x^2-a_3x+a_4),
  \end{equation*}
  where
  \begin{align*}
    a_1= & \,2q+p+r+4, \\
    a_2= & \,q^2+(p+r)q+pr+3(2q+p+r)+5, \\
    a_3= & \,2(q^2+pq+rq+pr+3q+p+r+1), ~~\text{and} \\
    a_4= & \,(p+q+r+2)q.
  \end{align*}
\end{theorem}
\begin{proof}
  If $p=0$ and $r=0$,
  then $G$ is isomorphic to the complete bipartite graph $K_{2,n-2}$. Thus
  \begin{equation*}
    \mu(G,x)=x(x-2)^{q-1}(x-q)(x-n).
  \end{equation*}
  We now consider the case $p \geq r \neq 0$.
  Label the vertices of $G$ as shown in Figure \ref{Fig:b-star}(a).
  Then the Laplacian matrix of $G$ has the form as in Figure \ref{Fig:L:B}.
  \begin{figure}[h!]
    \centering
    \includegraphics[width=9.5cm]{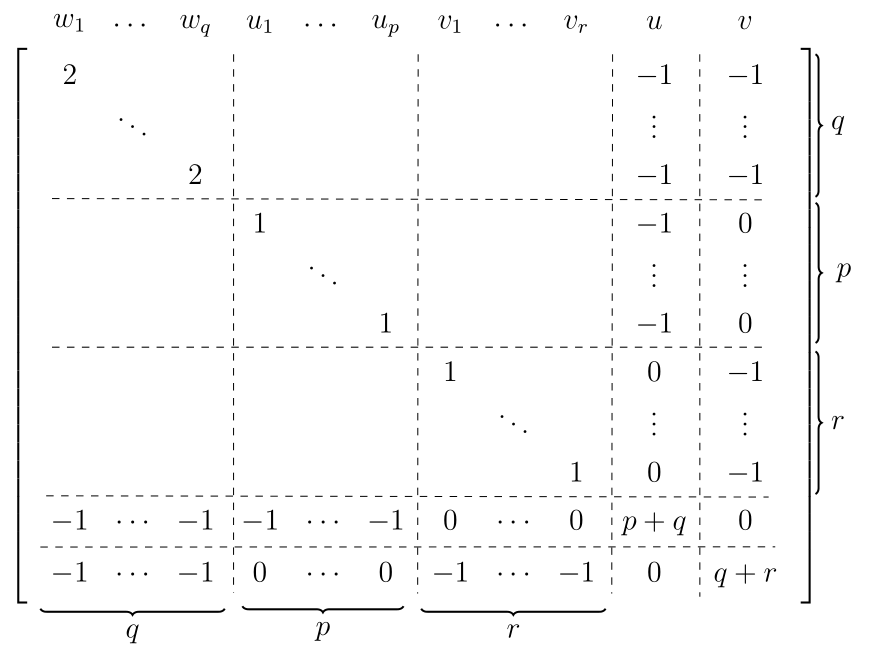}
    \caption{The Laplacian matrix of $B(p,q,r)$}\label{Fig:L:B}
  \end{figure}

  By Lemma \ref{stardeg} and Lemma \ref{deg2},
  the characteristic polynomial of $L(G)$ is of the form
  \begin{align*}
      \mu(G,x)=&\, x (x-1)^{p+r-2} (x-2)^{q-1} (x^4 -a_1x^3 +a_2x^2 -a_3x +a_4)\\
      =&\,x^n-(2q+p+r-4+a_1)x^{n-1}+\cdots+(-1)^{n-5}2^{q-1}a_4x.
    \end{align*}
  Since the trace of $L(G)$ is $4q+2p+2r$, we have
  \begin{equation*}
     a_1=2q+p+r+4.
  \end{equation*}
  In order to find $a_4$, we use the fact that the number of spanning trees multiplied by $n$ is equal to the product of all nonzero eigenvalues, which is $2^{q-1}a_4$.
  We proceed by counting the number of spanning trees of the binary star graph $B(p,q,r)$.
  Every spanning tree of $B(p,q,r)$ must contain
  \begin{enumerate}
    \item [(i)] all pendant edges,
    \item [(ii)] both of $uw_i$ and $w_iv$ for exactly one vertex $w_i$, and
    \item [(iii)] one of $uw_j$ or $w_jv$ but not both for $w_j\neq w_i$.
  \end{enumerate}
  The number of choices for (ii) is $q$ and the number of choices for (iii) is $2^{q-1}$.
  Hence there are $2^{q-1}q$ spanning trees.
  Therefore
  \begin{equation*}
    a_4=(p+q+r+2)q.
  \end{equation*}
  In order to find $a_2$ and $a_3$,
  we consider the matrices $L(G)-2I$ and $L(G)-I$.
  The characteristic polynomials of $L(G)-2I$ and $L(G)-I$ are
  \begin{equation*}
    \mu(L(G)-2I,x)=\mu(G,x+2)=x^n+\cdots+(2^5-2^4a_1+2^3a_2-2^2a_4+2a_4)x^{q-1}
  \end{equation*}
  and
  \begin{equation*}
    \mu(L(G)-I,x)=\mu(G,x+1)=x^n+\cdots+(-1)^{q-1}(1-a_1+a_2-a_3+a_4)x^{p+r-2},
  \end{equation*}
  respectively.

  Now, we compare the coefficient of $x^{q-1}$ in $\mu(L(G)-2I,x)$ and the sum of all $(n-(q-1))\times(n-(q-1))$ principal minors of $L(G)-2I$.
  Since the first $q$ rows of $L(G)-2I$ are identical and $n-(q-1)=p+r+3$, the principal minors that contribute to the sum have the following form:

  \begin{figure}[h!]
    \centering
    \includegraphics[width=9.5cm]{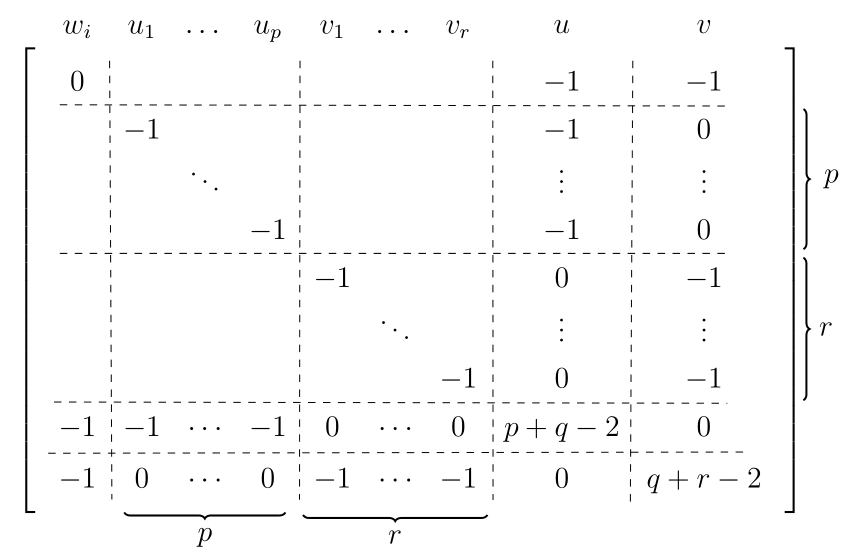}
  \end{figure}

  Since the determinant of the above matrix is $(-1)^{p+r+1}(2p+2r+2q-4)$, the sum of all the $(n-(q-1))\times(n-(q-1))$ principal minors of $L(G)-2I$ is $(-1)^{p+r+1}q(2p+2r+2q-4)$.
  Hence we obtain an equation
  \begin{equation}\label{eq1}
    2^5-2^4a_1+2^3a_2-2^2a_3+2a_4=q(2p+2r+2q-4).
  \end{equation}
  Now, we calculate the $(n-(p+r-2))\times(n-(p+r-2))$ principal minors of $L(G)-I$.
  By the same token, it is enough to consider the principal minors of the following form:
  \begin{figure}[h!t!]
    \centering
    \includegraphics[width=8cm]{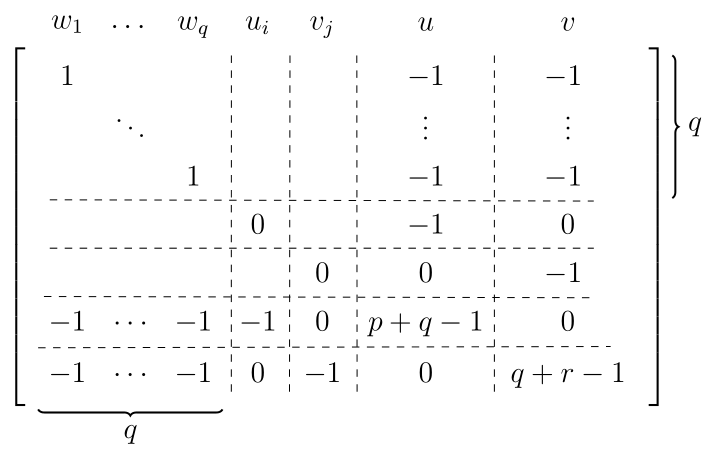}
  \end{figure}

  Since the determinant of the above matrix is $1$, the sum of all the $(n-(p+r-2))\times(n-(p+r-2))$ principal minors of $L(G)-I$ is $pr$, which gives
  \begin{equation}\label{eq2}
    1-a_1+a_2-a_3+a_4=pr.
  \end{equation}
  By the equations (\ref{eq1}) and (\ref{eq2}), we obtain
  \begin{equation*}
    a_2=q^2+(p+r)q+pr+3(2q+p+r)+5
  \end{equation*}
  and
  \begin{equation*}
    a_3= 2(q^2+pq+rq+pr+3q+p+r+1).
  \end{equation*}

  Finally, suppose that $p>0$ and $r=0$.
  Then the Laplacian matrix $L(G)$ is of the following form:
  \begin{figure}[h!t!]
    \centering
    \includegraphics[width=8cm]{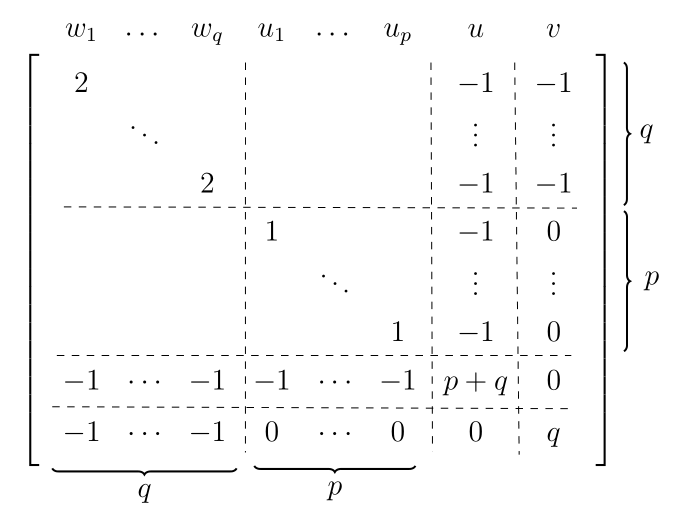}
  \end{figure}

  By Lemma \ref{stardeg} and \ref{deg2}, the characteristic polynomial of $L(G)$ is of the form
  \begin{align*}
    \mu(G,x)=\,&x(x-1)^{p-1}(x-2)^{q-1}(x^3-b_1x^2+b_2x-b_3)\\
    =&\,x^n-(2q+p-3+b_1)x^{n-1}+\cdots+(-1)^{n-3}2^{q-1}b_3x.
  \end{align*}
  Since the trace of $L(G)$ is $4q+2p$, we have
  \begin{equation*}
    b_1=2q+p+3.
  \end{equation*}
  Since the number of spanning trees of $B(p,q,0)$ is $2^{q-1}q$,
  \begin{equation*}
    b_3=(p+q+2)q.
  \end{equation*}
  In order to find $b_2$, we consider the matrix $L(G)-2I$.
  The characteristic polynomial of $L(G)-2I$ is
  \begin{equation*}
    \mu(L(G)-2I,x)=\mu(G,x+2)=x^n-\cdots+2(2^3-2^2b_1+2b_2-b_3)x^{q-1}.
  \end{equation*}
   By the similar method used in the previous case, the sum of all the $(n-(q-1))\times(n-(q-1))$ principal minors of $L(G)-2I$ is $(-1)^{p+1}q(2q+2p-4)$. Thus
  \begin{equation*}
    b_2=q^2+pq+4q+2p+2.
  \end{equation*}
  Since
  \begin{equation*}
    x^4-a_1x^3+a_2x^2-a_3x+a_4=(x-1)(x^3-b_1x^2+b_2x-b_3),
  \end{equation*}
  the proof is complete.
\end{proof}

\begin{theorem}\label{char:B'}
  Let $G$ be the binary star graph $B'(p,q,r)$ of order $n=p+q+r+2$. Then the characteristic polynomial of $L(G)$ is
  \begin{equation*}
    x(x-1)^{p+r-2}(x-2)^{q-1}(x^4 -a'_1x^3 +a'_2x^2 -a'_3x +a'_4),
  \end{equation*}
  where
  \begin{align*}
    a'_1= & \,2q+p+r+6, \\
    a'_2= & \,q^2+(p+r)q+pr+4(2q+p+r)+13, \\
    a'_3= & \,2(q^2+pq+rq+pr)+5(2q+p+r)+12, ~~\text{and}\\
    a'_4= & \,q^2+(p+r)q+2(2q+p+r)+4.
  \end{align*}
\end{theorem}
\begin{proof}
  The proof is similar to that of Theorem \ref{char:B}.
\end{proof}
\begin{remark}\label{spantree}
  For the number of spanning trees of the binary star graph $B'(p,q,r)$, we consider two cases.
  There are $2^{q-1}q$ spanning trees not containing the edge $uv$ and there are $2^q$ spanning trees containing the edge $uv$.
  Therefore the total number of spanning trees of $B'(p,q,r)$ is $2^{q-1}q+2^q$.
\end{remark}
\begin{remark}
  The binary star graph $B'(p,0,r)$ with $q=0$ is isomorphic to the double star $S(p,r)$.
  In this case, Theorem \ref{char:B'} recovers Lemma \ref{char:d-star}.
\end{remark}
\begin{theorem}\label{thm:main}
  Let $G$ be a connected graph of order $n$ with $\alpha(G)=n-2$.
  Then $m_G[0,n-\alpha(G)]=\alpha(G)$ if and only if  $G$ is isomorphic to one of the following forms:
  \begin{enumerate}
    \item $B(p,q,r)$ for $p+q+r \geq 3$,
    \item $B'(p,q,r)$ for $q=0$ and $pr\geq2$, or
    \item $B'(p,q,r)$ for $q\neq 0$.
  \end{enumerate}
\end{theorem}
\begin{proof}
  We point out that we are assuming $p \geq r$.
  Suppose that $m_G[0,2]=n-2$, which is equivalent to $m_G(2,n]=2$.
  Let $V(G)\backslash\{u,v\}$ be an independent set of $G$.
  As mentioned before, there are two cases:
  \medskip

  \noindent\textbf{Case 1.} $B(p,q,r)$: $u$ and $v$ are not adjacent.

  Since $G$ is connected, $q$ cannot be zero.
  Suppose that $p+q+r <3$.
  Then the graph $B(0,1,0)=P_3$ has independence number $n-1=2$. The graphs $B(1,1,0)=P_4$ and $B(0,2,0)=C_4$ have exactly one Laplacian eigenvalue greater than 2.
  Thus we have a contradiction.

  We assume $p+q+r \geq 3$.
  First, we consider the case for $p=0$.
  By Theorem \ref{char:B}, the characteristic polynomial of $L(B(0,q,0))$ is
  \begin{equation*}
    x(x-2)^{q-1}(x-q)(x-n).
  \end{equation*}
  Thus $B(0,q,0)$ has two Laplacian eigenvalues greater than $2$ for $q \geq 3$.
  Now, we assume that $p\neq0$.
  Let $f(x)$ be the quartic polynomial factor of $\mu(B(p,q,r),x)$ in Theorem \ref{char:B}.
  Then
  \begin{equation*}
    \begin{split}
      &f(0)= q(p+q+r+2)>0, \\
      &f(1)= -pr\leq0, \\
      &f(2)=q(p+q+r-2)>0, \\
      &f(p+q+1)= -q(q+(p-r))<0,
     \end{split}
  \end{equation*}
  and
  \begin{multline*}
     f(p+q+2)=(p-r)^3+(p-r)^2(2q+2r+3) \\
    +(p-r)(q^2+q(2r+3)+r^2+4r+2)+r(2q+r+2)>0.
  \end{multline*}
  If $r=0$, then $1$ is a root of $f$ and the other three roots of $f$ lie in $(0,1)\cup(1,2)$, $(2,p+q+1)$ and $(p+q+1,p+q+2)$.
  If $r \neq 0$, then the roots of $f$ lie in $(0,1)$, $(1,2)$, $(2,p+q+1)$ and $(p+q+1,p+q+2)$.
  Moreover, $\lambda_3(B(p,q,r))=2$ and $$2< \lambda_2(B(p,q,r)) <p+q+1 < \lambda_1(B(p,q,r)) <p+q+2.$$
  Thus the graph $B(p,q,r)$ has exactly two Laplacian eigenvalues greater than $2$.
  Therefore, $G$ is isomorphic to $B(p,q,r)$ for $p+q+r \geq 3$.
  \medskip

  \noindent\textbf{Case 2.} $B'(p,q,r)$: $u$ and $v$ are adjacent.

  \textbf{Case 2a.}
  We consider the case $q=0$.
  Suppose that $pr<2$.
  The graph $B'(p,0,0)=S_{p+2}$ has independence number $n-1=p+1$.
  The graph $B'(1,0,1)=P_4$ contains at most one Laplacian eigenvalue greater than 2. Thus we have a contradiction.

  Assume that $pr\geq 2$.
  Then $B'(p,0,r)$ is the double star $S(p,r)$.
  By Lemma \ref{char:d-star}, the characteristic polynomial of $L(S(p,r))$ is
  \begin{equation*}
    x(x-1)^{p+r-2}(x^3-(p+r+4)x^2+(2p+2r+pr+5)x-(p+r+2)).
  \end{equation*}
  Let $g(x)=x^3-(p+r+4)x^2+(2p+2r+pr+5)x-(p+r+2)$.
  Then we obtain
  \begin{equation*}
  \begin{split}
    &g(0)=-(p+r+2)<0, \\
    &g(2)=(p-r)(2r-1)+2r(r-1) > 0,\\
    &g(p+2)=-r< 0, ~~\text{and}\\
    &g(p+3)=(p-r)^2+(p-r)(r+4)+4>0.
  \end{split}
  \end{equation*}
  Hence $\lambda_3(B'(p,0,r))=2$ and $$2< \lambda_2(B'(p,0,r)) <p+2< \lambda_1(B'(p,0,r)) <p+3.$$
  Thus $G$ is isomorphic to $B'(p,0,r)$ for $q=0$ and $pr\geq2$.

  \textbf{Case 2b.} Suppose that $q\neq 0$.
  If $r=0$, by Theorem \ref{char:B'}, we have
  \begin{equation*}
    \mu(G,x)=x(x-1)^{p}(x-2)^{q-1}(x-(q+2))(x-n).
  \end{equation*}
  Hence $B'(p,q,r)$ has two Laplacian eigenvalues greater than $2$.
  Suppose that $r\neq0$.
  Let $h(x)$ be the quartic polynomial factor of $\mu(B'(p,q,r),x)$ in Theorem \ref{char:B'}.
  Then we have
  \begin{equation*}
  \begin{split}
    & h(0)= (q+2)(p+q+r+2)>0, \\
    & h(1)=  -pr<0, \\
    & h(2)= q(p+q+r)>0,\\
    & h(p+q+2)= -pr<0,
    \end{split}
  \end{equation*}
  and
  \begin{multline*}
    h(p+q+3)=(p-r)^3+(p-r)^2(2q+2r+5)\\
    +(p-r)(q^2+q(2r+6)+r^2+5r+8)+(q+2)(q+2r+2)>0.
  \end{multline*}
  By Theorem \ref{char:B'}, we have $\lambda_3(B'(p,q,r))=2$ and $$2< \lambda_2(B'(p,q,r)) <p+q+2< \lambda_1(B'(p,q,r)) <p+q+3.$$
  Thus $B'(p,q,r)$ has exactly two eigenvalues in the interval $(2,n]$.
  It implies that $G$ is isomorphic to $B'(p,q,r)$ for $q \neq 0$.
\end{proof}

\begin{corollary}
  The algebraic connectivity of $B(p,q,r)$ and $B'(p,q,r)$ is less than 1 for $pr \neq 0$.
\end{corollary}
\begin{proof}
  From the proof of Theorem \ref{thm:main}, we can deduce that the second smallest Laplacian eigenvalue is less than 1 for
  $B(p,q,r)$ and $B'(p,q,r)$ when $pqr \neq 0$. It is easy to see that the algebraic connectivity of $B'(p,0,r)$ is less than 1 for $pr\neq0$ from the fact that $g(1)=pr>0$, where
  $g(x)$ is the cubic polynomial factor of $\mu(B'(p,0,r),x)$.
\end{proof}

\section{Binary star graphs are DLS}

We say two graphs $G$ and $G'$ are {\it $L$-cospectral} if they have the same Laplacian spectrum. Clearly, two isomorphic graphs are $L$-cospectral. A classical question in spectral graph theory is whether the two $L$-cospectral graphs are isomorphic or not. A graph is said to be \emph{determined by its Laplacian spectrum}(DLS, for short) if there is no other non-isomorphic $L$-cospectral graph. In this section, we prove that when $p=r$, $B(p,q,p)$ and $B'(p,q,p)$ are DLS. We denote by \[ \deg(G)=(d_1, d_2, \ldots, d_n) \]
the nonincreasing degree sequence of a graph $G$. When there are repetitions, we use the superscripts to indicate the number of repetitions.
We collect some known results.

\begin{lemma}[\cite{Oliveira,vanDam}]\label{lem:Lcos}
For any graph, the following can be determined by its Laplacian spectrum.
\begin{enumerate}
  \item The number of vertices.
   \item The number of edges.
   \item The number of connected components.
   \item The number of spanning trees.
   \item The sum of the squares of degrees of vertices.
\end{enumerate}
\end{lemma}

\begin{lemma}[\cite{Grone-g2,Li-Pan,Guo}]\label{lem:Lspecdeg}
Let $G$ be a connected graph with $n\geq 4$ vertices. Then
\begin{enumerate}
\item $d_1\le \lambda_1(G) -1$.
\item $d_2\le \lambda_2(G) $.
\item $d_3\le \lambda_3(G)+1 $.
\end{enumerate}
\end{lemma}
\begin{lemma}[{\cite[Table 3]{HS}}]\label{lem:det:2}
  Any graph with at most $5$ vertices is DLS.
\end{lemma}
A tree is called {\it double starlike} if it has exactly two vertices of degree greater than two.
Denote by $H(p,n,q)$ the double starlike tree obtained by attaching $p$ pendant vertices to one pendant vertex of $P_n$ and $q$ pendant vertices to the other pendant vertex of $P_n$.
\begin{figure}[h!]
  \centering
  \includegraphics[width=8cm]{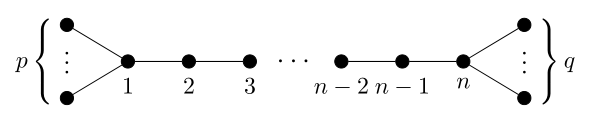}
  \caption{The double starlike tree $H(p,n,q)$ }\label{fig:d-starlike}
\end{figure}
\begin{lemma}[\cite{Liu,Lu}]\label{lem:det:1}
  Every double starlike tree $H(p,n,q)$ is DLS.
\end{lemma}
Thus, $B(p,1,r)=H(p,3,r)$ and $B'(p,0,r)=H(p,2,r)$ are DLS.
In what follows, we assume $q \geq 2$ for $B(p,q,p)$ and $q \geq 1$ for $B'(p,q,p)$.
In the following lemma, we show that any graph $L$-cospectral with $B(p,q,p)$ has the same degree sequence.
\begin{lemma}\label{lem:degseq}
Let $G= B(p,q,p)$ with $n=2p+q+2\ge 8$ and $q\ge 2$. Suppose that a graph $G'$ is $L$-cospectral with $G$. Then $$\deg(G')=\deg(G)=(\,p+q,\, p+q,\, 2^q,\, 1^{2p}\,).$$
\end{lemma}
\begin{proof}
We note that the condition $2p+q+2\ge 8$ and $q\ge 2$ implies $p+q\ge 4$.
Since $G'$ is $L$-cospectral with $G$, $G'$ has $n$ vertices and $2(p+q)$ edges. Let $(d_1, d_2, \ldots, d_n)$ be the nonincreasing degree sequence of $G'$. From the proof of Theorem \ref{thm:main}, we recognize that $p+q+1< \lambda_1(G)< p+q+2$ and $\lambda_3(G)=2$. Hence we have $d_1\le p+q$ and $d_3\le 3$ by Lemma \ref{lem:Lspecdeg}.

For $i=1,2,3$, let $n_i=|\{j : d_j= i, \,3\le j\le n \}|$. Then the degree sequence of $G'$ is $(d_1,\, d_2,\, 3^{n_3},\, 2^{n_2},\,1^{n_1})$. By Lemma \ref{lem:Lcos}, we have the following equations:
\begin{align}
2+ \sum_{i=1}^{3} n_i&= 2p+q+2 \label{v}\\
d_1+d_2+\sum_{i=1}^{3} in_i&= 4(p+q) \label{e}\\
d_1^2+d_2^2+\sum_{i=1}^{3} i^2n_i&= 2(p+q)^2+2p+4q \label{d2}
\end{align}
Combining \eqref{v}, \eqref{e} and \eqref{d2}, we have
\begin{equation}\label{bin2}
\binom{d_1-1}{2}+\binom{d_2-1}{2}+n_3 = 2\binom{p+q-1}{2}
\end{equation}
Plugging \eqref{bin2} back into \eqref{v} and \eqref{e} and solving for $n_2$, we have
\begin{equation}\label{n2}
n_2= d_1(d_1-4) + d_2(d_2-4) -2(p+q)(p+q-4) +q.
\end{equation}

We first show that $d_2\ge 3$. Suppose that $d_2<3$. Then $n_3=0$ and by \eqref{bin2}, we have $$ 2\binom{p+q-1}{2} =\binom{d_1-1}{2}+\binom{d_2-1}{2}\le \binom{p+q-1}{2}.$$ This inequality cannot hold when $ p+q\ge 4$. Thus, $d_1\ge d_2\ge 3$.

Now we show that $d_1=d_2= p+q$. Suppose that $d_1\le p+q-1.$ By \eqref{bin2} and \eqref{n2}, we have
\begin{align*}
n_3 &\ge  2\binom{p+q-1}{2} - 2\binom{p+q-2}{2}  = 2(p+q-2)~~\text{and}\\
n_2 &\le   2(p+q-1)(p+q-5) - 2(p+q)(p+q-4) +q = 10-4p-3q .
\end{align*}
The second inequality holds because $d_1\ge d_2\ge 3$.
Since $2(p+q-2)\leq n_3\le n-2=2p+q$ and $0 \leq n_2 \leq 10-4p-3q$, we obtain $q \leq 4$ and $4p+3q \leq 10$.
Hence $p+q \leq 3$ which contradicts to $p+q \geq 4$.
Therefore we conclude $d_1= p+q$.

Suppose that $d_2\le p+q-1$. By \eqref{n2}, we have
\[n_2 \le   (p+q-1)(p+q-5) - (p+q)(p+q-4) +q = 5-2p-q, \]
Since $2p+q+2 \geq 8$, $n_2<0$ which is a contradiction. Therefore $d_2= p+q$.

By \eqref{bin2}, we have $n_3=0$. Combining \eqref{v} and \eqref{e}, we obtain $n_2= q$ and $n_1= 2p$, which completes the proof.
\end{proof}

\begin{proposition}\label{prop:b}
The binary star graph $B(p,q,p)$ is DLS.
\end{proposition}

\begin{proof}
By Lemma \ref{lem:det:2} and \ref{lem:det:1}, it is enough to prove for $n=2p+q+2\ge 6$ and $q\ge 2$.
First, we assume $n=2p+q+2\ge 8$.
Let $G'$ be a graph $L$-cospectral with $B(p,q,p)$. By Lemma \ref{lem:degseq}, $G'$ has the degree sequence
\begin{equation}\label{eq:degG'}
  \deg(G')=\deg(G)=(\,p+q,\, p+q, \,2^q,\, 1^{2p}\,).
\end{equation}
Let $u$ and $v$ be vertices of $G'$ having the degree $p+q$. Suppose that $u$ and $v$ are adjacent. Let $s=|\{w\in V(G') : N_{G'}(w)=\{u,v\} \}|.$ By the degree condition (\ref{eq:degG'}), $u$ and $v$ have at least $2(p+q-1)-(2p+q)=q-2$ common neighbors. Thus $s\ge q-2$. Since $u$ and $v$ have degree $p+q$ and $G'$ is connected, there can be at most $2(p+q-1-s)$ pendant vertices. Hence we have $2(p+q-1-s)\ge 2p$. Therefore we have two cases $s=q-2$ and $s=q-1$.

When $s=q-2$, by the degree condition (\ref{eq:degG'}), two vertices in $(N_{G'}(u)\cup N_{G'}(v)) - (N_{G'}(u)\cap N_{G'}(v))$ must be adjacent. The possible two cases are shown in Figure \ref{fig:q-2}.
\begin{figure}[h!t!]
  \centering
  \includegraphics[width=13cm]{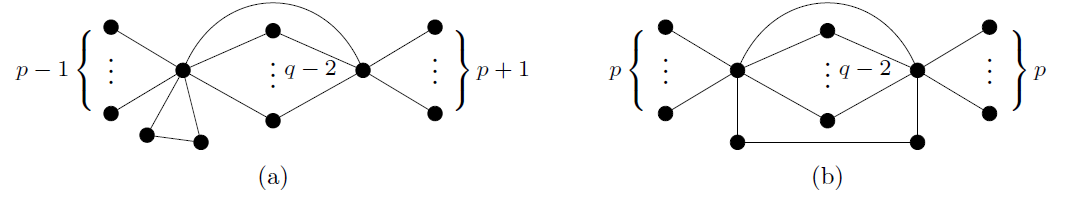}
  \caption{Graphs with $s=q-2$}\label{fig:q-2}
\end{figure}

Since $G'$ is $L$-cospectral with $B(p,q,p)$, the number of spanning trees of $G'$ is the same as that of $B(p,q,p)$, which is $2^{q-1}q$ by the proof of Theorem \ref{char:B}. But we can see that the number of spanning trees of the graph in Figure \ref{fig:q-2}(a) is $3(2^{q-2} +(q-2)2^{q-3})= 2^{q-3}3q$, and that of the graph in Figure \ref{fig:q-2}(b) is $2^{q-3}(3q+2)$.
Therefore we have a contradiction except for the graph in Figure \ref{fig:q-2}(b) when $q=2$.

Let  $G'$ be the graph in Figure \ref{fig:q-2}(b) for $q=2$.
\begin{figure}[h!t!]
  \centering
  \includegraphics[width=5cm]{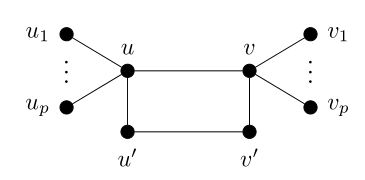}
  \caption{Graph in Figure \ref{fig:q-2}(b) for $q=2$ }\label{fig:q-2:q=2}
\end{figure}

\noindent We will show that $G'$ has Laplacian eigenvalue $1$ with multiplicity $2p-1$, which implies that $G'$ is not $L$-cospectral with $B(p,q,p)$, because the Laplacian eigenvalue $1$ of $B(p,q,p)$ has multiplicity $2p-2$.
We label the vertices as in Figure \ref{fig:q-2:q=2}.
When $p=1$, the matrix $L(G')-I$ is
\begin{figure}[h!t!]
  \centering
  \includegraphics[width=6cm]{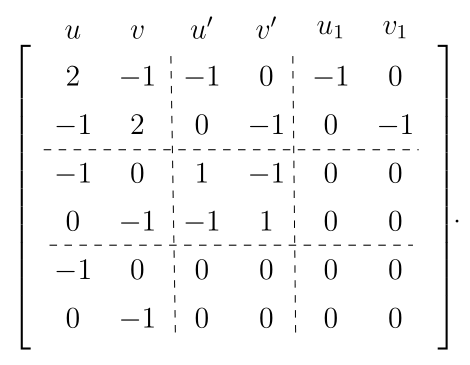}
\end{figure}

The rank of this matrix is $5$.
When $p>1$, the matrix $L(G')-I$ is obtained from the above matrix by adding the last two rows and columns $(p-1)$ times. Thus the matrix $L(G')-I$ still has rank $5$.
Hence the multiplicity of the Laplacian eigenvalue $1$ of $G'$ is $2p-1$.

Next, we examine the case $s=q-1$. When $s=q-1$, $G'$ is obtained by attaching a pendant vertex to a degree 1 vertex in $B'(p,q-1,p)$ as shown in Figure \ref{fig:q-1}.

\begin{figure}[h!t!]
  \centering
  \includegraphics[width=6cm]{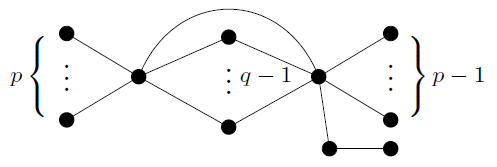}
  \caption{A graph with $s=q-1$}\label{fig:q-1}
\end{figure}

The number of spanning trees of graphs in this case is the same as that of $B'(p,q-1,p)$, which is $2^{q-2}(q+1)$ by Remark \ref{spantree} or by direct counting. This contradicts that $G'$ is $L$-cospectral with $B(p,q,p)$.

Now we suppose that $u$ and $v$ are not adjacent. In this case, by the degree condition \eqref{eq:degG'}, $u$ and $v$ must have at least $2(p+q)-(2p+q)=q$ common neighbors. Therefore $G'$ must be isomorphic to $B(p,q,p)$.

By Lemma \ref{lem:det:2} and \ref{lem:det:1}, the remaining cases are $n=6,7$ and $q\geq 2$: $(p,q)=(0,4), (1,2)$ when $n=6$ and  $(0,5), (1,3)$ when $n=7$.
By using the same argument as Lemma \ref{lem:degseq}, one can check all possible degree sequences satisfying \eqref{v}, \eqref{e} and \eqref{d2}.
We list here all the cases.
\begin{center}
\begin{tabular}{|c|c|c|}
  \hline
  $n$ & $(p,q)$ & degree sequence  \\
  \hline
  6 & $(1,2)$ & $(\,3,\,3,\,2,\,2,\,1,\,1\,)$ \\
  6 & $(0,4)$ & $(\,4,\,3,\,3,\,3,\,2,\,1\,)$  \\
  6 & $(0,4)$ & $(\,4,\,4,\,2,\,2,\,2,\,2\,)$  \\
  7 & $(1,3)$ & $(\,4,\,3,\,3,\,3,\,1,\,1,\,1\,)$  \\
  7 & $(1,3)$ & $(\,4,\,4,\,2,\,2,\,2,\,1,\,1\,)$  \\
  7 & $(0,5)$ & $(\,5,\,4,\,3,\,3,\,3,\,1,\,1\,)$  \\
  7 & $(0,5)$ & $(\,5,\,5,\,2,\,2,\,2,\,2,\,2\,)$ \\
  \hline
\end{tabular}
\end{center}
When $\deg(G')=(\,p+q,\,p+q,\,2^q,\,1^{2p}\,)$, the same argument as above shows that there are no graphs $L$-cospectral with $B(p,q,p)$.
For the other cases,
all connected graphs with their Laplacian eigenvalues(rounded to two decimal places) are shown in Figure \ref{fig1}, \ref{fig2} and  \ref{fig3}.
All of them are not $L$-cospectral with the corresponding $B(p,q,p)$, which completes the proof.
\end{proof}
\begin{figure}[h!t!]
  \centering
  \includegraphics[width=12cm]{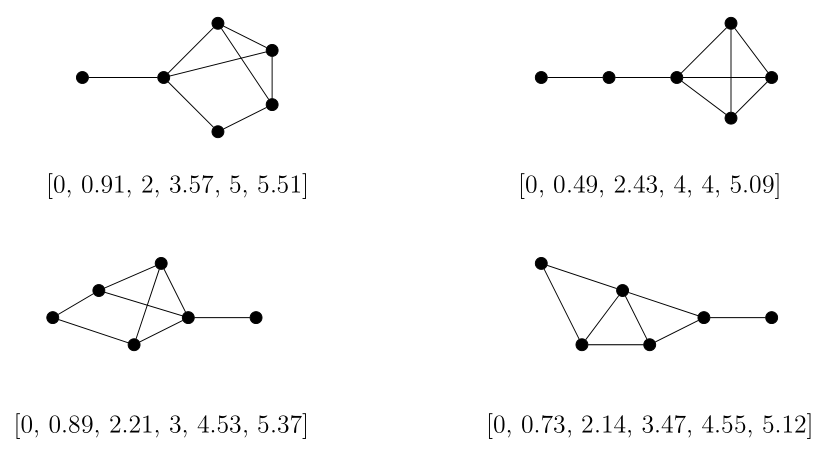}
   \caption{Graphs with degree $(\,4,\,3,\,3,\,3,\,2,\,1\,)$}\label{fig1}
\end{figure}
\begin{figure}[h!t!]
  \centering
  \includegraphics[width=12cm]{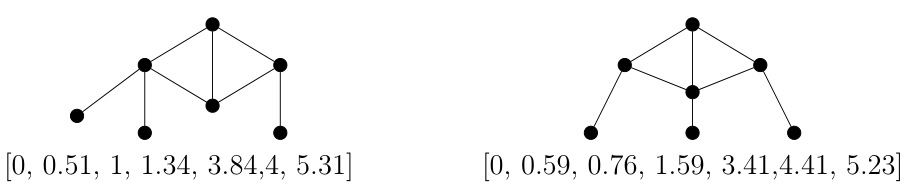}
  \caption{Graphs with degree $(\,4,\,3,\,3,\,3,\,1,\,1,\,1\,)$}\label{fig2}
\end{figure}

\begin{figure}[h!t!]
  \centering
  \includegraphics[width=12cm]{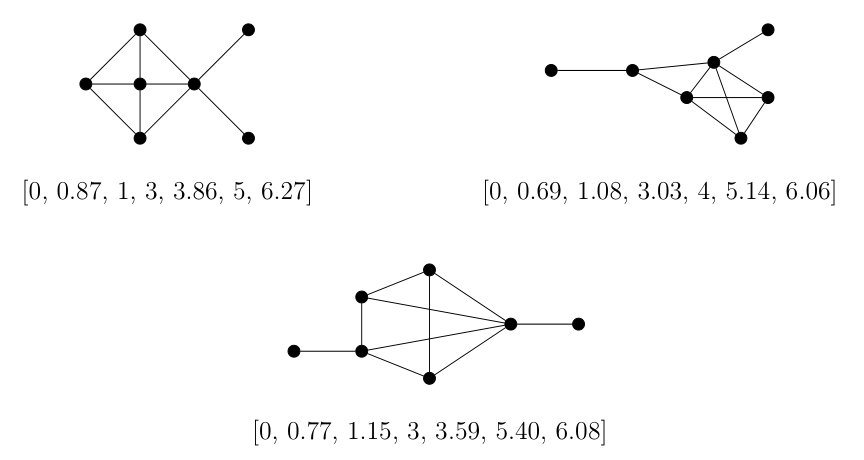}
  \caption{Graphs with degree $(\,5,\,4,\,3,\,3,\,3,\,1,\,1\,)$}\label{fig3}
\end{figure}

Now we show that $B'(p,q,p)$ is DLS.

\begin{lemma}\label{lem:degseq2}
Let $G= B'(p,q,p)$ with $2p+q+2\ge 6$ and $q\ge 1$. Suppose that a graph $G'$ is $L$-cospectral with $G$. Then $$\deg(G')=\deg(G)=(\,p+q+1, \,p+q+1, \,2^q, \,1^{2p}\,).$$
\end{lemma}
\begin{proof}
We note that the condition $n=2p+q+2\ge 6$ and $q\ge 1$ implies $p+q\ge 3$.
Since $G'$ is $L$-cospectral with $G$, $G'$ has $n$ vertices and $2(p+q)+1$ edges. Let $(d_1, d_2, \ldots, d_n)$ be the nonincreasing degree sequence of $G'$. From the proof of Theorem \ref{thm:main}, we recognize that $p+q+2 \leq \lambda_1(G)< p+q+3$ and $\lambda_3(G)=2$. Hence we have $d_1\le p+q+1$ and $d_3\le 3$ by Lemma \ref{lem:Lspecdeg}.

We use the same notation for $n_i$ as in Lemma \ref{lem:degseq}. Then the degree sequence of $G'$ is $(\,d_1,\, d_2, \,3^{n_3}, \,2^{n_2},\,1^{n_1}\,)$. By Lemma \ref{lem:Lcos}, we have the following equations:
\begin{align}
2+ \sum_{i=1}^{3} n_i&= 2p+q+2 \label{v1}\\
d_1+d_2+\sum_{i=1}^{3} in_i&= 4p+4q+2 \label{e1}\\
d_1^2+d_2^2+\sum_{i=1}^{3} i^2n_i&= 2(p+q+1)^2+2p+4q \label{d21}
\end{align}
Combining \eqref{v1}, \eqref{e1} and \eqref{d21}, we have
\begin{equation}\label{bin21}
\binom{d_1-1}{2}+\binom{d_2-1}{2}+n_3 = 2\binom{p+q}{2}
\end{equation}
Plugging \eqref{bin21} back into \eqref{v1} and \eqref{e1} and solving for $n_2$, we have
\begin{equation}\label{n21}
n_2= d_1(d_1-4) + d_2(d_2-4) -2(p+q+1)(p+q-3) +q.
\end{equation}

We first show that $d_2\ge 3$. Suppose that $d_2< 3$. Then $n_3=0$ and by \eqref{bin21}, we have $$ 2\binom{p+q}{2} =\binom{d_1-1}{2}+\binom{d_2-1}{2}\le \binom{p+q}{2}.$$ This inequality cannot hold when $p+q\ge 3$. Thus, $d_1\ge d_2\ge 3$.

Now we show that $d_1=d_2= p+q+1$. Suppose that $d_1\le p+q.$ By \eqref{bin21} and \eqref{n21}, we have
\begin{align*}
n_3 &\ge  2\binom{p+q}{2} - 2\binom{p+q-1}{2}  = 2(p+q-1) ~~\text{and}\\
n_2 &\le   2(p+q)(p+q-4) - 2(p+q+1)(p+q-3) +q = 6-4p-3q .
\end{align*}
Since $2(p+q-1)\leq n_3\le 2p+q$ and $0\leq n_2\leq 6-4p-3q$, we obtain $q \leq 2$ and $4p+3q \leq 6$.
Hence $p+q \leq 2$ which contradicts to $p+q\geq3$. Therefore we conclude $d_1= p+q+1$.

Suppose that $d_2\le p+q$. By \eqref{n21}, we have
\[n_2 \le   (p+q)(p+q-4) - (p+q+1)(p+q-3) +q = 3-2p-q. \]
Since $2p+q+2 \geq 6$, $n_2<0$ which is a contradiction. Therefore $d_2= p+q+1$.

By \eqref{bin21}, we have $n_3=0$. By \eqref{v1} and \eqref{e1}, we obtain $n_2= q$ and $n_1= 2p$, which completes the proof.
\end{proof}

\begin{proposition}\label{prop:bbar}
The binary star graph $B'(p,q,p)$ is DLS.
\end{proposition}
\begin{proof}
By Lemma \ref{lem:det:2} and \ref{lem:det:1}, we may assume $2p+q+2\ge 6$ and $q\ge 1$.
Let $G'$ be a graph $L$-cospectral with $B'(p,q,p)$. By Lemma \ref{lem:degseq2}, $G'$ has the degree sequence
\begin{equation}\label{eq:degG':B'}
  \deg(G')=\deg(G)=(\,p+q+1, \,p+q+1,\, 2^q,\, 1^{2p}\,).
\end{equation}
Let $u$ and $v$ be vertices of $G'$ having the degree $p+q+1$. Suppose that $u$ and $v$ are not adjacent. Then by the degree condition \eqref{eq:degG':B'}, $u$ and $v$ must have at least $2(p+q+1)-(2p+q)=q+2$ common neighbors, which contradicts that there are $q$ degree 2 vertices. Thus $u$ and $v$ are adjacent.

When $u$ and $v$ are adjacent, by the degree condition \eqref{eq:degG':B'}, $u$ and $v$ must have $q$ common neighbors. Then it is easy to see that $G'$ is isomorphic to $B'(p,q,p)$.
\end{proof}


\end{document}